\documentclass[11 pt]{article}
\usepackage[top=2cm, bottom=2cm, right=2cm, left=2cm]{geometry}
\usepackage[UKenglish]{babel}
\usepackage{latexsym,amsfonts,amsmath,amsthm,amssymb,mathrsfs}
\usepackage{todonotes}
\usepackage{xfrac}
\usepackage{comment}
\newcommand{\supp}{\operatorname{supp}}
\usepackage{bm}
\usepackage{dsfont}
\usepackage[colorlinks=true, linkcolor=teal, citecolor=-teal]{hyperref}
\usepackage{setspace}
\usepackage{etoolbox}
\usepackage[skip=2pt]{parskip}
\usepackage{titlesec}
\newcommand{\aaa}{c}

\titlespacing*{\section}
{0pt}{0.5\baselineskip}{\baselineskip}
\titlespacing*{\subsection}
{0pt}{0.5\baselineskip}{0.5\baselineskip}

\patchcmd{\thebibliography}{\section*{\refname}}{}{}{}
\DeclareMathOperator{\R}{\mathbb{R}}
\DeclareMathOperator{\N}{\mathbb{N}}

\makeatletter
\@namedef{subjclassname@2010}{%
	\textup{2010} Mathematics Subject Classification}
\makeatother

\newtheorem{thm}{Theorem}
\newtheorem{Theorem}{Theorem}
\newtheorem{lem}{Lemma}[section]
\newtheorem{Lemma}[lem]{Lemma}

\newtheorem{Definition}{Definition}

\newtheorem{Proposition}[lem]{Proposition}

\newtheorem*{rem*}{Remark}

\begin{document}

\newtheorem{conj}{Conjecture}
\newtheorem{prob}{Open Problem}

\title{A century of metric Diophantine approximation\\ and half a decade since Koukoulopoulos--Maynard}

\author{Manuel Hauke}
\date{}

\maketitle

\begin{abstract}
    In this note, we review the history of Khintchine's Theorem which is the foundation of metric Diophantine approximation, and discuss several generalizations and recent breakthroughs in this area. We focus particularly on the direction of the Duffin--Schaeffer Conjecture which was spectacularly proven in 2020. We present some simplified key ideas of the proof that can also be applied in various other areas of number theory.
\end{abstract}

\section{A century of Khintchine's Theorem in metric Diophantine approximation}\label{intro}
The original aim of Diophantine approximation is to provide good rational approximations to irrational numbers and to estimate the error made by this approximation. The history of this area goes back to ancient times when engineers tried to approximate $\pi$ by rational approximations such as $22/7$. Although quite classical in its nature, the area of Diophantine approximation is a lively area of modern interest; it is not only a very active field in pure mathematics with spectacular breakthroughs in recent years such as \cite{EKL6,KMDS}, but it is also of great importance in computer sciences: Since a computer has only a finite memory capacity, a precise numerical evaluation of an irrational number is impossible, and a rational approximation has to be considered - the smaller the denominator of the approximating rational number, the less memory needs to be allocated. The larger the denominator, the better the approximation we can expect, so there is a trade-off between memory allocation and rounding error made, and optimal solution(s) to this problem can be found quite efficiently by using the theory of Diophantine approximation.\\

To state the problem described above in a mathematical manner, we consider an approximation function $\psi: \N \to [0,\infty)$ and for given real (typically irrational) $\alpha$, one asks for rationals $p/q \in \mathbb{Q}$ such that
\[\left\lvert\alpha - \frac{p}{q}\right\rvert \leq \frac{\psi(q)}{q}.\] The function $\psi$ can be seen as a penalizing factor that tries to avoid large denominators unless the rate of approximation turns out to be much better when using a larger denominator.\\
While the name is derived from the ancient Greek mathematician Diophantus of Alexandria (ca. 3th century) who was one of the first to consider related questions, the origin of Diophantine approximation in the sense we understand it today is Dirichlet's
Theorem (ca. 1830). The Theorem states (in its non-uniform variant, sometimes also referred to as \textit{Dirichlet's Corollary}) that for every real number $\alpha$, we find infinitely many rational numbers $p/q$ that satisfy $\left\vert\alpha - \frac{p}{q}\right\vert \leq \frac{1}{q^2},$ i.e. we can choose $\psi(q) = \frac{1}{q}$ in the above framework. Although the proof is elementary and concise, this approximation rate is, up to a constant, optimal when considering arbitrary irrational numbers. Numbers where Dirichlet's Theorem is sharp are called \textit{badly approximable numbers} and include prominent examples such as the Golden Ratio $\Phi = \frac{1 + \sqrt{5}}{2}, \sqrt{2}$, or in general, all irrational numbers that arise as an irrational real solution of a quadratic polynomial equation with rational coefficients.\\

However, not all irrational numbers are badly approximable, with one prominent example being Euler's number $e$. In fact, ``most'' numbers are not badly approximable, i.e. drawing a number uniformly at random from $[0,1)$\footnote{We can restrict our analysis without loss of generality to $\alpha \in [0,1)$ since the integer part of a real number does not affect the Diophantine properties of $\alpha$.}, we expect with full probability to beat Dirichlet's approximation by a little more than a factor of $\log q$. This was beautifully established a centenary ago by Khintchine, and the statement is considered today as the starting point of the theory on \textit{metric} Diophantine approximation:\\

\noindent {\bf Khintchine's Theorem (1924):} {\it
Let $\psi: \N \to [0,\infty)$ be monotonically decreasing.
Writing
\[ A(\psi) := \left\{\alpha \in [0,1): \left\lvert \alpha - \frac{p}{q} \right\rvert \leq \frac{\psi(q)}{q} \text{ for infinitely many } (p,q) \in \mathbb{Z} \times \N \right\},
\]
we have that

\[
\lambda(A(\psi))= \begin{cases}
0 &\text { if }  \sum\limits_{q \in \N} \psi(q) < \infty,\\
    1 &\text{ if } \sum\limits_{q \in \N}\psi(q) = \infty,
\end{cases}
\]
where $\lambda$ denotes the ($1$-dimensional) Lebesgue measure.}\\

We remark that this theorem was established during a time when axiomatic probability theory was about to be introduced by Kolmogorov and further developed by (the very same) Khintchine. Khintchine's Theorem can be seen from (at least) three different perspectives:\\

\begin{itemize}
    \item 

In probability-theoretic language, we have a probability space $\Omega = [0,1)$ with probability measure $\mathbb{P} = \lambda_{[0,1)}$, events $E_q := \bigcup_{p \in \mathbb{Z}}\left\{\alpha \in [0,1): \left\lvert \alpha - \tfrac{p}{q} \right\rvert \leq \tfrac{\psi(q)}{q}\right\}$, and want to determine $\mathbb{P}[\limsup_{q \to \infty}{E_q}]$.\\

\item 
From a dynamics point of view, we are given the one-dimensional torus $\mathbb{T} = \mathbb{R}/\mathbb{Z} \cong [0,1)$, and want to understand at which speed the dynamical system
$T_{\alpha}^n(0)$ driven by the irrational rotation $T_{\alpha}(x) := x + \alpha \pmod 1$ approaches $0$. Note that $\lambda$ is the probability measure that makes $T_{\alpha}$ a measure-preserving map.\\

\item  Since the Lebesgue measure intuitively can be seen as the correct way of depicting to choose a number ``randomly'', the (number-theoretic) interpretation is the above-mentioned improvement rate over Dirichlet's Theorem for ``typical'' $\alpha$.\\
\end{itemize}

Of course, these are just different perspectives and clearly ``three sides of the same coin''.
However, for the proof respectively generalizations below, it can be beneficial to view these problems from one perspective or another.
We remark on the similarity to Borel's result on normal numbers \cite{B09} where Borel proved that ``almost every number is normal in any integer base''. There, Borel was strictly separating the probabilistic and number-theoretic points of view by seeing them as two separate \textit{statements}, but with the axiomatic probability theory being introduced about two decades later, the mathematical community realized that these are just two \textit{interpretations} of the same statement.

\subsection{Generalizations of Khintchine's Theorem}

Since the development of Khintchine's Theorem, researchers tried to generalize the occurring phenomenon in various setups. Some natural considerations are the following, with the possibility of connecting various generalizations:

\begin{enumerate}
    \item What happens when we replace the Lebesgue measure with some other probability measure $\mu$ that is supported on $[0,1)$? As long as there is no good reason for this measure to be ``maliciously chosen'', the general philosophy is that Khintchine's Theorem should still hold true with $\mu$ in place of the Lebesgue measure. It is obvious that for measures $\mu$ that are absolutely continuous with respect to the Lebesgue measure will still satisfy Khintchine's Theorem (since sets of measure zero remain sets of measure zero), so for non-trivial questions, we have to look for different measures. Note that we can restrict our analysis also to non-atomic measures since otherwise, Khintchine's Theorem will be automatically wrong. A class of measures that gained a lot of interest, especially in recent years are the uniform measures on fractals such as missing-digit sets like the middle-third Cantor set, following a question raised by Mahler \cite{M84}. There have been various impressive results in the last few years (e.g. \cite{CVY24,DJ24,KL23,Y21}) with the question being recently settled in another breakthrough result \cite{BHZ24}.
    The proof methods are mainly tools from homogeneous dynamics (combined with harmonic analysis), a field of research that is closely connected to (classical) Diophantine approximation.
    \item Closely related to the above is the study of Khintchine's Theorem on manifolds: Taken a manifold $\mathcal{M} \subseteq \R^n$ that (at least locally) can be described as $F(U)$ where $U \subseteq \R^d$ is an open set and $F$ is a reasonably smooth map, we use the normalized Lebesgue measure on $U$ to define a probability measure $\mu$ on $\mathcal{M}$, and ask about a Khintchine Theorem with respect to $\mu$ on the manifold.

    This turns out to be a challenging problem: While algebraic solutions can help to understand rational points \textit{on} manifolds, the theory of Diophantine approximation on manifolds is closely related to counting points \textit{near} manifolds, making it naturally a question of an analytic flavour. For some breakthrough results in this area, we refer to e.g. \cite{BDV07,BY23,KM98,SST23,VV06}.\\
    We remark that although Diophantine approximation on manifolds might at first sight look like a purely mathematical problem with barely any real-world applications, there is the possibility to apply these concepts in wireless communication - we refer the interested reader to the gentle introduction in \cite{BV20}.
    
    \item Introducing an inhomogeneous parameter: Instead of asking for $\left\lvert \alpha - \frac{p}{q} \right\rvert \leq \frac{\psi(q)}{q}$, we ask for infinitely many solutions to \[\left\lvert \alpha - \frac{p + \gamma}{q} \right\rvert \leq \frac{\psi(q)}{q},\]
    where $\gamma \neq 0$ is an arbitrary fixed real number. The inhomogeneous generalization of Khintchine's Theorem was proven by Sz\"usz in 1958 \cite{S58}, but this is still a very active area of research (for some recent advancements see e.g. \cite{BHV24,CT24,Y_inhom}). From a dynamical point of view, inhomogeneous Diophantine approximation means we examine the visits of a neighbourhood of zero of $T_{\alpha}^n(\gamma)$ in place of $T_{\alpha}^n(0)$. While na\"{i}vely one might expect this to not make any difference, in some cases this parameter complicates matters: For the homogeneous case, the classical theory of continued fractions, i.e. the representation of $\alpha$ as 
    \[\alpha = \frac{1}{a_1 + \frac{1}{a_2 + \frac{1}{a_3 + \ldots}}}, \quad a_i \in \mathbb{N},\]
    can be successfully used to prove Khintchine's Theorem and various generalizations. Since there is no direct analogue to continued fractions known for inhomogeneous approximations, one has to appeal to other methods.\\
    
    One can even try to push the game and ask for a ``moving target'', i.e. we take a different $\gamma$ for each $q$ and ask about $\left\lvert \alpha - \frac{p + \gamma_q}{q} \right\rvert \leq \frac{\psi(q)}{q}$ for infinitely many $q$. One of these cases arises when we try to connect different dynamical systems, such as e.g.
    $T_{\alpha}^n(T_{\beta}^{f(n)}(0))$ where say $\beta$ is fixed and $f: \N \to \N$ is an arbitrary function. While for specific dynamical systems, an exploration might yield a Khintchine phenomenon, Khintchine's Theorem with an arbitrary moving target is still open \cite{HR24}.\\
    \item Taking the probability-theoretic point of view again, Khintchine's Theorem is a statement on the measure of the $\limsup$-set of a sequence of events. If the events were a family of independent events, the statement would follow from a simple application of the Borel--Cantelli Lemma. However, these sets do not satisfy this independence condition, but the original Borel--Cantelli Lemma was over the years refined to allow some weak form of dependence that still allows us to deduce the $0-1$-result. On a more technical note (see Section \ref{var_as} for details), one tries to control the degree of dependence by getting upper bounds on the variance, i.e. \begin{equation}\label{variance_def}\sum_{q,r \leq Q} \lambda(E_q \cap E_r) - \sum_{q \leq Q}\lambda(E_q), \quad \text{ when } Q \to \infty.\end{equation} While for perfect pairwise independence, only the diagonal ($q = r$) gives a non-zero contribution, refined versions of the Borel--Cantelli-Lemma allow (much) weaker bounds for \eqref{variance_def} to be sufficient in order to deduce the $0-1$-result. Researchers tried over the years to strengthen bounds on the variance of the sets in Khintchine's Theorem, and examined the degree of dependence of the events $(E_q)_{q \in \N}$, in the hope of obtaining ``sufficient independence'' to establish other classical results of probability theory that are known to hold in the case of independent events\footnote{Of course, even in the independent case, there are restrictions imposed on the corresponding measures of the events in order to satisfy some of the mentioned statements - we will omit these details here for the sake of conciseness.}. In particular, the following questions have been studied in detail:
    \begin{itemize}
        \item Can we establish a ``Strong law of large numbers'', i.e. do we have for (Lebesgue-) almost every $\alpha \in [0,1)$ that \begin{equation}\label{asympototic_form}\sum_{q \leq Q} \mathds{1}_{E_q}(\alpha) \sim \sum_{q \leq Q}\lambda(E_q), \quad \text {as } Q \to \infty?\end{equation}
        If this holds true, what can we say about the error term (i.e. the difference) arising? The first who established \eqref{asympototic_form} for Khintchine's Theorem was Schmidt \cite{S60}, and this result has since then been adapted to various related setups, see e.g.\cite{ABH23,AG22,CT22,DKL05,H23,P24,S64}.
        \item Can we find a Central Limit Theorem, i.e., do we have
        \[\lambda\left(\left\{\alpha \in [0,1): \frac{\sum_{q \leq Q} \mathds{1}_{E_q}(\alpha) - \sum_{q \leq Q}\lambda(E_q)}{
\sqrt{\sum_{q,r \leq Q} \lambda(E_q \cap E_r) - \sum_{q \leq Q}\lambda(E_q)}}
        \right\}\right) \stackrel{w}{\longrightarrow} \mathcal{N}(0,1) \text{ as } Q \to \infty?\]
        This was proven by \cite{F03,L_I,L_II} for Khintchine's Theorem (see also \cite{BG19} for a more general approach) under some regularity assumptions on $\psi$.
        \item Do we have a law of the iterated logarithm? This was partially established by Fuchs \cite{F04}, by demanding some further assumptions on $\psi$.
    \end{itemize}
    \item Removing monotonicity:  Although having a monotonicity assumption in $\psi$ for Khintchine's Theorem seems natural from the point of view of ``penalizing'' large denominators, there are, among other reasons, two interesting settings where dealing with a non-monotonic $\psi$ is unavoidable:
\begin{itemize}
    \item Multiplicative Diophantine approximation: 
    Motivated by the famously open conjecture of Littlewood, a natural question is asking about having infinitely many solutions to\footnote{Here and in what follows, $\lVert . \rVert$ stands for the distance to the nearest integer.} 
    \begin{equation}\label{mult_DA}q \lVert q \beta \rVert \lVert q \alpha \rVert \leq \theta(q)\end{equation} for some $\theta: \N \to [0,\infty)$ (which itself can be thought of as being decreasing). Littlewood's Conjecture, which claims that for all $\varepsilon > 0$ and $\theta(q) = \varepsilon$, \eqref{mult_DA} admits infinitely many solutions, is known to be true (by quite some margin) for almost every pair $(\alpha,\beta)$ \cite{G62}, and even the Hausdorff dimension of exceptional pairs is known to be zero \cite{EKL6}. However, many questions (such as several of the probabilistic refinements described above) remain unsolved.\\
    Note that we can rewrite the inequality in the multiplicative setup (under the natural irrationality assumption on $\beta$) by
    $\lVert q \alpha \rVert \leq \frac{\theta(q)}{q\lVert q \beta \rVert} =: \psi_{\beta}(q)$, an approach successfully applied in e.g. \cite{CT22,CT24,FH24}.
    Clearly, $\psi_{\beta}(q)$ is typically a non-monotonic function, and thus, Khintchine's Theorem cannot be applied. 
    \item Studying Diophantine approximation with a restriction on the allowed denominators: Here, we consider an increasing integer sequence $(a_n)_{n \in \mathbb{N}}$ and suppose that we only want to allow approximations with denominators coming from this sequence. For real-world applications, such a restriction can come for example from numbers that one can store efficiently in a computer (e.g. numbers whose binary expansion has only a few non-zero entries). From a purely mathematical point of view, the question translates to the metric behaviour of $(a_n\alpha)_{n \in \N}$ mod $1$, an area that has a very rich history in number theory, and was considered already by many illustrious researchers, such as Erd\H{o}s \cite{E49,E64}, Hardy and Littlewood \cite{HL46}, Weyl \cite{W16} and numerous others.
    Here, $\psi(q) := \mathds{1}_{[q \in (a_n)_{n \in \N}]}\theta(q)$ for some $\theta: \N \to [0,\infty)$ (which itself can be thought of as being monotonically decreasing) provides the according model of a non-monotonic $\psi$.\\
\end{itemize}
\end{enumerate}

While one could write book series about each of the above-mentioned generalizations, we will focus in the remainder of this article on the question of removing the monotonicity assumption from Khintchine's Theorem.
In their seminal work. Duffin and Schaeffer \cite[1941]{DS41} showed that the monotonicity assumption in Khintchine's Theorem is necessary. The Counterexample made heavily use of the trivial fact that a rational number has various distinct representations as fractions of integers, e.g. \[1/2 = 2/4 = 3/6 = \ldots.\] The consequence of this trivial observation is that if $\alpha$ is very close to 
$\frac{a}{b}$, then it is also close to
$\frac{am}{bm}$ for every $m \in \N$. If the values of $\psi$ are now distributed in a way that their values are particularly big on multiples of $b$, then this messes up the usually predicted ``random'' behaviour that occurs when the values of $\psi$ are evenly distributed among the integers (which follows from the monotonicity assumption on $\psi$). In order to circumvent this, Duffin and Schaeffer suggested to count every fraction only once in its reduced form, i.e. we consider 
the question of whether the equation 
\[\left\lvert \alpha - \frac{p}{q} \right\rvert \leq \frac{\psi(q)}{q}, \quad \gcd(p,q) = 1\] is satisfied infinitely often. Counting arguments (see Section \ref{var_as}) straightforwardly imply that the criterion for full or empty measure now should rather be whether or not 
$\sum\limits_{q \in \N} \frac{\psi(q)\varphi(q)}{q} < \infty$, where $\varphi$ denotes the Euler totient function. The corresponding Conjecture was open for 78 years and after partial, but yet important contributions over the years such as 
\cite{ALMZ,BHHV,BV_Mass,E70,G61,PV90}, was finally resolved in a highly celebrated breakthrough article by Koukoulopoulos and Maynard \cite{KMDS} in 2020, which is seen as one of the reasons J. Maynard was awarded with the Fields Medal in 2022.
\\

\begin{thm}[Duffin--Schaeffer Conjecture 1942/Koukoulopoulos--Maynard Theorem 2020]
    Let $\psi: \N \to [0,\infty)$ be an arbitrary function and let
    \[A'(\psi) :=
\left\{\alpha \in [0,1): \left\lvert \alpha - \frac{p}{q} \right\rvert \leq \frac{\psi(q)}{q}, \gcd(p,q) = 1 \text{ for infinitely many } p,q \in \mathbb{N}\right\}.\]

Then we have that
\[
\lambda\left(A'(\psi)\right) = \begin{cases}
0 &\text { if }  \sum\limits_{q \in \N} \left(\frac{\varphi(q)\psi(q)}{q}\right) < \infty,\\
    1 &\text{ if } \sum\limits_{q \in \N} \left(\frac{\varphi(q)\psi(q)}{q}\right) = \infty.
\end{cases}
\]
\end{thm}

The above-mentioned counterexample of Duffin and Schaeffer \cite{DS41} already hints at a connection to asking about greatest common divisors and related number-theoretic concepts. 
On a more technical note (see the proof sketch in Section \ref{sec_proof_sketch}), looking at the variance estimate reveals that the ``enemy'' for the solution of the Duffin--Schaeffer Conjecture are functions $\psi$ where the support lies on numbers that share unusually often a large GCD. More precisely, a critical example will have an unusually large (weighted) GCD sum
\[\sum_{q,r \leq Q}\frac{\varphi(q)\psi(q)}{q}\frac{\varphi(r)\psi(r)}{r}\frac{\gcd(q,r)}{\sqrt{qr}}.\]

We refer the interested reader to the excellent monograph of \cite[Chapter 3]{H98} for a more detailed connection between GCD sums and Duffin--Schaeffer-type questions.\\

Therefore, the problem (and finally the solution) attracted rather researchers from analytic number theory than researchers working on questions of Diophantine approximation with tools from homogeneous dynamics. The reason for this is that the usually applied dynamical tools are mostly incompatible with the concept of non-monotonic approximation functions.

\section{A proof sketch for the Duffin--Schaeffer conjecture}\label{sec_proof_sketch}

In the second part of this note, we provide a simplified proof sketch for the Duffin--Schaeffer Conjecture. We remark that is impossible to mention all articles that have contributed to the final resolution of the conjecture. In this sketch, the main concepts will come from Pollington and Vaughan \cite{PV90} respectively Kouloulopoulos and Maynard \cite{KMDS}. However, we present the recent 
proof given by the author with Vazquez and Walker \cite{HSW}
that simplified, shortened, and strengthened the original proof of Kouloulopoulos and Maynard\footnote{This should not take anything away from the original proof. On the contrary, we believe that it is the duty of the mathematical community to streamline and simplify the proofs of important results as much as possible in order to make it understandable for as many researchers as possible, and should therefore rather highlight the author's admiration of the ingenious proof strategy developed by Kouloulopoulos and Maynard.}. This approach builds on the work of Green and Walker \cite{GW21} which solved a problem in combinatorial number theory on the number of pairs with large GCD.\\

For simplicity, we will assume that $\psi: \N \to [0,1/2]$. The case where $\psi(q) > 1/2$ for infinitely many $q$ has been resolved by Pollington--Vaughan in 1990 \cite{PV90} by much shorter arguments. Further, we will assume that $\psi$ is only supported on square-free numbers. While the proof for the general case has some technical tweaks to include the non-squarefree case, the key ideas and concepts are still contained in this  sketch under this simplifying assumption.\\
    
The proof is split into the following parts:\\

\begin{enumerate}
    \item \textbf{From variance estimate to almost sure}:\\
    In this part, we will show that the problem can, analogously to Khintchine's Theorem, be translated into a question on the Lebesgue measure of $\limsup_{q \to \infty} A_q$ for some particular measurable sets $A_q$. We will further show that the Duffin--Schaeffer Conjecture can be reduced to the question of proving the variance estimate
    \begin{equation}\label{QIA}\exists\, C > 1:\forall 0 < Y < X: \sum_{X \leq q \leq Y}\lambda(A_q) \in [1,2] \implies \sum_{X < q < Y}\lambda(A_q \cap A_r) \leq C\cdot\left(\sum_{X < q < Y}\lambda(A_q)\right)^2.\end{equation}
    This can be shown by a combination of Borel--Cantelli-type Lemmas with a $0-1$-law established in 1961 \cite{G61}, and therefore is a question of measure/probability theory (combined with ergodic theory), and does not contain the particularly difficult part of the proof. We refer for details to Section \ref{var_as}.\\
    
    \item \textbf{The overlap estimate - application of sieve theory}:
    
    This part deals with upper bounds on $\frac{\lambda(A_q \cap A_r)}{\lambda(A_q)\lambda(A_r)}$ - if this quantity was bounded by a constant $C$ for all pairs $q \neq r$, this would immediately show \eqref{QIA}. However, it will turn out to be too much to ask for. In this section, most of the work has been developed by Erd\H{o}s \cite{E70} respectively Pollington and Vaughan \cite{PV90}. This part of the proof uses, besides elementary arguments, mainly sieve theory, which is a classical topic in analytic number theory. On a more technical note, the mentioned authors used upper-bound sieves that allow to lose a (multiplicative) constant factor - which is negligible in view of \eqref{QIA}. Such sieves are way more flexible and allow a much shorter sieve range in comparison to sieves that are pushing towards an asymptotic. For the quantitative versions of the Duffin--Schaeffer Conjecture in the sense of \eqref{asympototic_form} subsequently proven \cite{ABH23,HSW,KMY24}, new sieve estimates had to be developed. \\
    \item \textbf{The combinatorial part - the non-existence of a set with only large GCDs}:\\
    This is the heart of the proof and the key to the breakthrough result of Koukoulopolus--Maynard. It will be stated below as Proposition \ref{Proposition:Prop5.4} when the corresponding quantities are defined. This part of the proof  (see Section \ref{comb_part}) can be seen as a result in combinatorial number theory. The proof of \eqref{QIA} and consequently the Duffin--Schaeffer Conjecture follows from Proposition \ref{Proposition:Prop5.4} quite straightforwardly after a ``dyadic splitting'' (see Section \ref{DS_from_prop}).

\subsection{From variance estimate to almost sure}\label{var_as}

The Duffin--Schaeffer Conjecture makes (as mentioned for Khintchine's Theorem on p.\pageref{variance_def}) an assertion about the measure of the limsup-set
$\limsup_{q \to \infty}A_q$ for a sequence of events $(A_q)_{q \in \N}$. In contrast to the set system $(E_q)_{q \in \N}$ that underlies Khintchine's Theorem, the additional coprimality assumption shows that the sets $A_q$ are given by 

\begin{equation}\label{def_Aq}A_q = \bigcup_{\substack{0 \leq a \leq q - 1\\\gcd(a,q) = 1}}\left[\frac{a}{q} - \frac{\psi(q)}{q}, \frac{a}{q} + \frac{\psi(q)}{q}\right].\end{equation}

Computing the Lebesgue measure of these sets (which can be seen as computing the expected value with respect to the probability measure $\mathbb{P} = \lambda_{[0,1]}$), we obtain (recall we assume $\psi(q) \leq 1/2$ throughout the article)
\[\lambda(A_q) = \sum_{\substack{0 \leq a \leq q - 1\\\gcd(a,q) = 1}} \frac{2\psi(q)}{q} = 2  \frac{\varphi(q)\psi(q)}{q}.\]

Hence if $\sum_{q \in \N } \frac{\varphi(q)\psi(q)}{q} < \infty$, then $\sum_{q \in \N }\lambda(A_q) < \infty$, so by the convergence Borel--Cantelli Lemma, it follows immediately that if
$\lambda(\limsup_{q \to \infty} A_q) = 0$, which proves the convergence part of the Duffin--Schaeffer Conjecture immediately.\\

Hence we are left to prove that
\[\sum_{q \in \N } \frac{\varphi(q)\psi(q)}{q} = \infty \implies
\lambda(\limsup_{q \to \infty} A_q) = 1.\] As described in Section \ref{intro} for Khintchine's Theorem, we would ideally like to apply the Divergence Borel--Cantelli Lemma, but since the events $(A_q)_{q \in \N}$ are not mutually independent, we cannot do that. The assumption of mutual independence however can be weakened by only assuming ``quasi-independence on average'' in the following sense:\\

\begin{Lemma}[Refined divergence Borel--Cantelli Lemma]\label{dbc}
 \! Let $(X,\mathcal{A},\mu)$ be a probability space  and let $\{E_i\}_{i \in \N} $ be a sequence of subsets in $\mathcal{A}$.
Suppose that
\begin{equation}\label{eqn01}
\sum_{i=1}^\infty \mu(E_i)=\infty
\end{equation}
and that there exists a constant $C>0$ such that
\begin{equation}\label{eqn02}
\sum_{s,t=1}^Q  \mu(E_s\cap E_t)\le C\left(\sum_{s=1}^Q  \mu(E_s)\right)^2\quad\text{for infinitely many $Q\in\N$\,.}
\end{equation}
Then
\begin{equation}\label{eq3}
\mu(E_{\infty}) \ge C^{-1}\,.
\end{equation}
\end{Lemma}

The ideas for this result date back to articles of Paley and Zygmund in the 1930's \cite{PZ30,PZ32} arising essentially from the usage of the Cauchy--Schwarz inequality, and were made more explicitly by Chung and Erd\H{o}s \cite{CE52} (see \cite{BV23} for a recent survey). Lemma \ref{dbc} weakens the assumption on $(A_q)_{q \in \N}$ from mutual independence to proving \eqref{eqn02} with $C = 1$. While the latter has been established 3 years after the first proof of the Duffin--Schaeffer Conjecture in \cite[Theorem 2]{ABH23}, it considerably simplifies the proof if we showed that it suffices to prove \eqref{eqn02} for some \textit{arbitrary} $C > 1$. The $0-1$-law of Gallagher \cite{G61} for the Duffin--Schaeffer conjecture comes in very handy: It states that regardless of the function $\psi$, we always have for the particular set system \eqref{def_Aq} that
\begin{equation}\label{01}\lambda(\limsup_{q \to \infty}{A_q}) \in \{0,1\}.\end{equation} Together with Lemma \ref{dbc}, this implies that it indeed suffices to show \eqref{eqn02} for some $C > 1$: By \eqref{eq3}, we obtain $\lambda(\limsup_{q \to \infty}{A_q}) \ge \frac{1}{C} > 0$, so by \eqref{01} we immediately deduce $\lambda(\limsup_{q \to \infty}{A_q}) = 1$. We remark that Gallagher's $0-1$-law makes use of the Lebesgue Density Theorem as well as ergodic arguments, and is thus analytically non-trivial and clearly restricted to the usage of Lebesgue measure as well as the particular set system \eqref{def_Aq}. If one wants to apply this strategy to more general systems or prefers a more elementary proof, we refer the reader to \cite{BHV24} where more general versions of the refined divergence Borel--Cantelli-Lemmas are proven. We appeal to the following formulation from \cite[Theorem 6]{BHV24} for the streamlined proof of the Duffin--Schaeffer Conjecture:\\

\begin{thm}\label{T2MS}
Let $\mu$ be a doubling Borel regular probability measure on a metric space $X$. Let $\{E_i\}_{i\in\N}$ be a sequence of $\mu$-measurable subsets of $X$. Suppose that there exist constants $C>0$ and $\aaa>0$ and {a sequence of finite subsets $S_k\subset\mathbb{N}$ such that $\min S_k\to+\infty$ as $k\to\infty$, and such that} 
\begin{equation}\label{eqn04}
\sum_{i\in S_k}\mu(E_i) \ge \aaa
\end{equation}
and
\begin{equation} \label{eqn05}
\sum_{\substack{s<t\\[0.5ex] s,t\in S_k}} \mu\big(E_s\cap E_t \big) \ \le \  C\,  \left(\sum_{i\in S_k}\mu(E_i)\right)^2
\end{equation}

 In addition, suppose that for any $\delta>0$ and any closed ball $B$ centered at $\supp\mu$
\begin{equation}\label{vb89}
\limsup_{\substack{i \to \infty\\ \mu(E_i) > 0}}\frac{\mu\left(B\cap E_i\right)}{\mu\left(B\right)\mu(E_i)}\le 1.
\end{equation}
Then
$\mu(\limsup_{i \to \infty}E_i)=1$.
\end{thm}

The proof of Theorem \ref{T2MS} itself is quite elementary, but for the sake of conciseness, will nevertheless be omitted here.
Note that in the case of the set system \eqref{def_Aq}, it follows immediately (by the equidistribution of $\{0 \leq a \leq q-1: a/q : \gcd(a,q) = 1\}$ in $[0,1]$ as $q \to \infty$) that 
\eqref{vb89} is satisfied for $\mu = \lambda_{[0,1]}$, so it suffices to show \eqref{eqn05} for some $C > 1$, avoiding completely the usage of a $0-1$-law.\\
The sets $S_k$ in \ref{T2MS} for the application to our problem will be taken to be of the form of increasing, disjoint integer intervals  $[X,Y]$ with 
$\sum_{X \leq q \leq Y}\lambda(A_i) \in [1,2]$. Therefore, the Duffin--Schaeffer Conjecture follows immediately as soon as we can prove the following, which is equivalent to \eqref{QIA}:\\

\begin{Proposition}\label{QIA_prop}
Let $\psi: \N \to [0,1/2]$ be an arbitrary function and let $1 < X < Y$ be such that 
\[\sum_{\substack{q\in [X,Y]}}\lambda(A_q) \in [1,2].\]
Then there exists $C > 0$ (independent of $X,Y$) such that 
\[\sum_{\substack{X \leq q,r \leq Y}}\lambda(A_q) \lambda(A_r) \leq C.\]
\end{Proposition}

Thus we have reduced the proof to showing \eqref{QIA}.

\subsection{The sieve estimate}

We follow the general ideas of Pollington and Vaughan \cite{PV90}. Given $q,r \in \N$ with $q\neq r$, we aim to understand the so-called \textit{overlap}, i.e. $\lambda(A_q \cap A_r)$, in the form of a good upper bound in terms of  $\lambda(A_q)\lambda(A_r)$.
Looking at \eqref{def_Aq}, we see that $A_q$ consists of small intervals of length $\frac{\psi(q)}{q}$ around rationals $a/q$ where $a$ is coprime to $q$. Therefore any $x \in A_q \cap A_r$ has to lie in
\[\left[\frac{a}{q} - \frac{\psi(q)}{q}, \frac{a}{q} + \frac{\psi(q)}{q}\right] \cap \left[\frac{b}{r} - \frac{\psi(r)}{r}, \frac{b}{r} + \frac{\psi(r)}{r}\right]\] for some $a \in \mathbb{Z}_q^{*}, b \in \mathbb{Z}_r^{*}$. Observe that $\left[\frac{a}{q} - \frac{\psi(q)}{q}, \frac{a}{q} + \frac{\psi(q)}{q}\right] \cap \left[\frac{b}{r} - \frac{\psi(r)}{r},  \frac{b}{r} + \frac{\psi(r)}{r}\right]$ is empty unless 
\[\left\lvert\frac{a}{q} - \frac{b}{r}\right\rvert \leq 2\max\left\{\tfrac{\psi(q)}{q},\tfrac{\psi(r)}{r}\right\},\]
and since the length of the intersection of two intervals is trivially bounded by the length of the shorter one, we have
\[\lambda\left(\left[\frac{a}{q} - \frac{\psi(q)}{q}, \frac{a}{q} + \frac{\psi(q)}{q}\right] \cap \left[\frac{b}{r} - \frac{\psi(r)}{r},  \frac{b}{r} + \frac{\psi(r)}{r}\right]\right) \leq 2\min\left\{\tfrac{\psi(q)}{q},\tfrac{\psi(r)}{r}\right\}.\]

Assuming now by symmetry arguments $\tfrac{\psi(q)}{q} \leq \tfrac{\psi(r)}{r}$, we obtain\footnote{We use the Vinogradov notation $f \ll g \Leftrightarrow f = O(g)$.}
\[\lambda(A_q \cap A_r) \ll 
\tfrac{\psi(q)}{q}\#\left\{(a,b) \in \mathbb{Z}_q^{*} \times \mathbb{Z}_r^{*}: \left\lvert\tfrac{a}{q} - \tfrac{b}{r}\right\rvert \leq \tfrac{\psi(r)}{r}\right\}.
\]

Thus we are left with a counting problem, which after the application of the Chinese Remainder Theorem and a simple counting argument shows that we have\footnote{In this article, a summation respectively a product over $p$ is always understood to be a summation respectively product over prime numbers.}

\[\#\left\{(a,b) \in \mathbb{Z}_q^{*} \times \mathbb{Z}_r^{*}: \left\lvert\tfrac{a}{q} - \tfrac{b}{r}\right\rvert \leq \tfrac{\psi(r)}{r}\right\} = 
2 \tfrac{\varphi(\gcd(q,r))^2}{\gcd(q,r)} \sum_{\substack{1 \leq c \leq \frac{q\psi(r)}{\gcd(q,r)}\\(c,\frac{qr}{\gcd(q,r)^2}) = 1}} \prod_{p \mid (\gcd(q,r),c)} \left(1 + \frac{1}{p-1}\right).
\]

The right-hand side above is now in the form of a typical sieve problem: In general, many sieve problems are of the form
$\sum_{\substack{1 \leq c \leq X\\(c,n) = 1}} f(c)$
where $f$ is a weight function that very often (as here) is \textit{multiplicative}, i.e. $f(mn) = f(m)f(n)$ for $\gcd(m,n) = 1$. Such a weight function only causes minor inconveniences, so what one essentially tries to prove is that
\[\sum_{\substack{1 \leq c \leq X\\(c,n) = 1}} 1 \approx X \frac{\varphi(n)}{n},\] since the proportion of all integers that are coprime to $n$ is $\tfrac{\varphi(n)}{n}$. This heuristic is of course far from the truth when $X$ is particularly small for a trivial reason: Assume that $p \mid n$ for a prime such that $p > X$. Then a number $1 \leq c \leq X$ is coprime to $n$ if and only if it is coprime to $n/p$,
so \[\sum_{\substack{1 \leq c \leq X\\(c,n) = 1}} 1 = \sum_{\substack{1 \leq c \leq X\\(c,n/p) = 1}} 1.\]

Note however that $\frac{\frac{\varphi(n)}{n}}{\frac{\varphi(n/p)}{n/p}}= 1 - \frac{1}{p}$; The same argumentation can be done for any prime $p > X$, so the best we could hope for is to achieve
\[\sum_{\substack{1 \leq c \leq X\\(c,n) = 1}} 1 \ll X \frac{\varphi(n)}{n} \left(1 + \prod\limits_{\substack{p > X\\ p \mid n}}\left(1 - \frac{1}{p}\right)^{-1}\right).\]

Indeed, the \textit{combinatorial sieve} (see e.g. \cite[Chapter 2]{sieves}) achieves this for our problem (with $X = \frac{q\psi(r)}{\gcd(q,r)},n = \frac{qr}{\gcd(q,r)^2}$).
Simplifying by elementary means and defining 
\begin{equation}\label{D_def}D = D(q,r) := \frac{\max(q\psi(r), r \psi(q))}{\gcd(q,r)},\end{equation} we obtain

\begin{equation}\label{overlap_final}\begin{split}\lambda(A_q \cap A_r) &\ll \frac{2\varphi(q)\psi(q)}{q}\frac{2\varphi(r)\psi(r)}{r}\prod_{\substack{p \mid qr\\p > q\psi(r)}} \left(1 - \frac{1}{p}\right)^{-1}
\\&\ll \lambda(A_q) \cdot\lambda(A_r)\prod_{\substack{p \mid \frac{qr}{\gcd(q,r)^2}\\p > D(q,r)}} \left(1 + \frac{1}{p}\right).
\end{split}
\end{equation}

\subsection{Proof of the combinatorial part}\label{comb_part}

If \[\prod\limits_{\substack{p \mid \frac{qr}{\gcd(q,r)^2}\\p > D(q,r)}} \left(1 + \frac{1}{p}\right) \asymp \exp\Bigg(\sum\limits_{\substack{p \mid \frac{qr}{\gcd(q,r)^2}\\p > D(q,r)}}
\frac{1}{p}\Bigg)\] was uniformly bounded by a constant, then the proof would follow immediately from Steps 1 and 2 (in the form of combining Proposition \ref{QIA_prop} with \eqref{overlap_final}). However, since $\sum\limits_{p \in \mathbb{P}} \frac{1}{p}$ diverges, this term can in general be arbitrarily big, which means that a pointwise point by simply considering fixed pairs $(q,r)$ is doomed to fail. Thus we need to control $\prod\limits_{\substack{p \mid qr\\p > q\psi(r)}} \left(1 + \frac{1}{p}\right)$ on \textit{average} over $q,r$ with respect to the bilinear weights $w_{q,r} := 
\frac{\varphi(q)\psi(q)}{q}\frac{\varphi(r)\psi(r)}{r}$ that arise from $\lambda(A_q) \cdot\lambda(A_r).$ However, $\psi$ is an \textit{arbitrary} function whose support might be super-sparse, so it seems almost impossible (at least with common averaging techniques known in analytic number theory\label{avg_techn}) to get a useful improvement by averaging - which however is necessary. This is the slightly technical explanation why the Duffin--Schaeffer Conjecture was open for 79 years and needed the invention of a completely novel technique for its resolution. We want to remark that this strategy has not only been applied since its invention for results in metric Diophantine approximation \cite{ABH23,HSW,KMY24}, but has since been adapted to achieve results in combinatorial number theory \cite{GW21} as well as solving a Conjecture of Erd\H{o}s from 1948 on integer dilation approximations \cite{KLL25}. We believe this strategy to be useful in various future applications across different areas of number theory. This is the main reason why we provide an oversimplified view of the key ideas of the proof in the next subsection, with the hope of making the proof more attainable for a broader audience.\\
Returning to the proof of the Duffin--Schaeffer Conjecture, the technical main achievement \cite[Proposition 5.4]{KMDS}) can be rephrased by the following proposition.\\

\begin{Proposition}\label{Proposition:Prop5.4}
    Let $0 \leq X \leq Y$ and $\psi: \mathbb{N} \cap [X,Y] \to [0,\infty)$ such that
    \[1 \leqslant \sum_{X \leqslant n \leqslant Y}\frac{\psi(n)\varphi(n)}{n} \leqslant 2.\]
    For each $j \in \N$, denote
    \begin{equation}
        \label{def_ej}
    \mathcal{E}_{e^j} := \Big\{(v,w) \in (\mathbb{N} \cap [X,Y])^2: D(v,w) \leqslant e^j, \sum_{\substack{p \vert \frac{vw}{\gcd(v,w)^2} \\ p \geqslant e^j}} \frac{1}{p}\geqslant 10 \Big\}.\end{equation}
    Then 
    \[\sum_{(v,w) \in \mathcal{E}_{e^j}} \frac{\psi(v)\varphi(v)}{v}\frac{\psi(w)\varphi(w)}{w} \ll \frac{1}{e^j}.\]
\end{Proposition}

Before we provide the proof of Proposition \ref{Proposition:Prop5.4}, we show how to finish the proof of the Duffin--Schaeffer Conjecture from here.

\subsubsection{Proof of the Duffin--Schaeffer conjecture assuming Proposition \ref{Proposition:Prop5.4}}\label{DS_from_prop}
Given a pair $(q,r)$, we define $j(q,r)$ to be the smallest $j$ such that
$\sum_{\substack{p \mid \frac{qr}{(q,r)^2}\\p > e^j}}\frac{1}{p}\geq 10$. 
This provides a partition
$[X,Y]^2 = \bigcup_{j \in \N}\mathcal{F}_j$
where $\mathcal{F}_j := \{(q,r) \in [X,Y]^2: j(q,r) = j\}$.

For a pair $(q,r) \in \mathcal{F}_j$, we see that by the classical Mertens estimate
\[\sum_{y \leq p \leq x} \frac{1}{p} = \log \log x - \log \log y + O(1),\]
we get
\[\begin{split}\prod_{\substack{p \mid \frac{qr}{\gcd(q,r)^2}\\p > D(q,r)}} \left(1 + \frac{1}{p}\right)
&\ll \exp\Bigg(\sum_{\substack{p \mid \frac{qr}{\gcd(q,r)^2}\\ D(q,r)< p < e^{j}}}\frac{1}{p} 
+ \underbrace{\sum_{e^j < p < e^{j+1}}\frac{1}{p}}_{= O(1)} + 
\underbrace{\sum_{\substack{p \mid \frac{qr}{\gcd(q,r)^2}\\p > e^{j+1}}} \frac{1}{p}}_{\leq 10}\Bigg)
\\&\ll \exp\Bigg(O(1) + \sum_{\substack{p \mid \frac{qr}{\gcd(q,r)^2}\\ D(q,r)< p < e^{j}}}\frac{1}{p} \Bigg)
\\&\ll  \exp\Bigg(\sum_{\substack{p \mid \frac{qr}{\gcd(q,r)^2}\\ D(q,r)< p < e^{j}}}\frac{1}{p} \Bigg)
\ll \begin{cases}
    1 \text{ if } D(q,r) > e^j,\\
    j \text { otherwise}.
\end{cases}
\end{split}\]

In view of \eqref{overlap_final}, this shows

\[\sum_{(q,r) \in \mathcal{F}_j}\lambda(A_q \cap A_r) \ll \sum_{(q,r) \in \mathcal{F}_j}\lambda(A_q) \lambda(A_r)
+ j \sum_{\substack{(q,r) \in \mathcal{F}_j\\D(q,r) \leq e^j}}\lambda(A_q) \lambda(A_r).
\]

Now we note that any pair $(q,r) \in \mathcal{F}_j$ that satisfies $D(q,r) \leq e^j$ is contained in $\mathcal{E}_{e^j}$ as in \eqref{def_ej}, so an application of Proposition \ref{Proposition:Prop5.4} proves 

\[\sum_{\substack{(q,r) \in \mathcal{F}_j\\D(q,r) \leq e^j}}\lambda(A_q) \lambda(A_r) \ll \frac{1}{e^j}.\]
Thus we obtain
\[\sum_{(q,r) \in \mathcal{F}_j}\lambda(A_q \cap A_r) \ll \frac{j}{e^j} + \sum_{(q,r) \in \mathcal{F}_j}\lambda(A_q) \lambda(A_r).
\]
Summing over all $j$, and using that $\sum_{j \in \N}\frac{j}{e^j}$ converges, shows
\[\sum_{\substack{(q,r) \in [X,Y]^2}}\lambda(A_q) \lambda(A_r) \ll 1,\]
which by Proposition \ref{QIA_prop} proves the Duffin--Schaeffer Conjecture.

\begin{proof}[Proof of Proposition \ref{Proposition:Prop5.4}]

\label{strategy}
Here we follow the strategy of \cite{HSW} that can be summarized as follows: Assuming that there exists a counterexample with a minimal number of primes involved, i.e., given an integer interval $[X,Y]$, let $S(\psi) := \supp(\psi) \cap [X,Y]$ and let $P(\psi) := \{p \in \mathbb{P}: \exists q \in S(\psi): p \mid q\}$, and assume that for $(P(\psi),\psi)$ Proposition \ref{Proposition:Prop5.4} is false, but for any $\Theta: \N \to [0,\infty)$ with $P(\Theta) \subsetneq P(\psi)$, the proposition is true.\\

We show that in this case (see Proposition \ref{Prop:structureofminimalcounterexample}), most of the ``weight'' $\frac{\varphi(q)\psi(q)}{q}\frac{\varphi(r)\psi(r)}{r}$ is concentrated (multiplicatively) around one integer $N$, in the sense that most of the weight is concentrated on pairs $(q,r)$ such that 
\begin{equation}\lvert \nu_p(N/q)\rvert + \lvert \nu_p(N/r)\rvert \leq 1.\label{padic_structure}\end{equation} Here, $\nu_p$ denotes the $p$-adic valuation, i.e. for $a/b \in \mathbb{Q}$,
\[\nu_p(a/b) := \max \{j \in \N: p^j \mid a\} - \max\{j \in \N: p^j \mid b\}.\]

Establishing \eqref{padic_structure} is the main step in the proof. The second step (Proposition \ref{Prop:resolvingminimalcounterexample}) is then showing that on this set, condition \eqref{padic_structure} makes it impossible for this to be a counterexample. This provides a contradiction to the assumption that a minimal counterexample exists, which implies that no counterexample can exist.\\

 It turns out that it is \textit{easier} to prove Proposition \ref{Proposition:Prop5.4} in a more general statement; while this sounds counterintuitive, we need this to get a stronger ``induction hypothesis'' for the minimal counterexample procedure to work\footnote{In this form, the approach of \cite{HSW} which we consider here, differs from the one of Koukoulopoulos--Maynard \cite{KMDS} (and is the main reason for the simplification): In \cite{KMDS}, no such generalization is considered, but one defines a delicate iterative procedure where one seeks to get a number-theoretic structure similar to \eqref{padic_structure} by maximizing an ad-hoc defined \textit{quality} that in the end (seemingly miraculously) does the job.}. 
 
 In order to state the more general statement, we  need to introduce some notation.
 We work on pairs $(q,r) \in [X,Y]^2$, where a weight is given of the form $w_{q,r} := \frac{\varphi(q)\psi(q)}{q}\frac{\varphi(r)\psi(r)}{r}$. This can be translated in the language of graph theory where the edge set of the graph is given by $\mathcal{E} \subset V \times W$ with $V,W \subseteq [X,Y]$, and an edge $(v,w) \in \mathcal{E}$ gets assigned the bilinear weight $w_{q,r} := \frac{\varphi(q)\psi(q)}{q}\frac{\varphi(r)\psi(r)}{r}$. 

 The generalization means that we allow asymmetric weights of the form $w_{q,r} := \frac{\varphi(q)\psi(q)}{q}\frac{\varphi(r)\theta(r)}{r}$ where $\theta,\psi: \N \to \mathbb{R}_{\geq 0}$ are possibly distinct weight functions with support in $[X,Y]$.\\

\begin{Definition}[Measures]
Let $\psi,\theta: \mathbb{N} \longrightarrow [0,\infty)$ be finitely supported. For $v \in \mathbb{N}$ and $V \subset \N$, we define \[\mu_{\psi}(v): = \frac{\varphi(v) \psi(v)}{v}, \quad \mu_{\psi}(V) := \sum_{v \in V} \mu_{\psi}(v).\] If $\mathcal{E} \subset \mathbb{N} \times \mathbb{N}$, we define \[ \mu_{\psi,\theta}(\mathcal{E}): = \sum_{(v,w) \in \mathcal{E}} \mu_{\psi}(v) \mu_{\theta}(w),\]
and denote by $\Gamma_{\mathcal{E}}(v): = \{w: (v,w) \in \mathcal{E}\}$ the neighbours of $v$ in the graph $G = (V \cup W, \mathcal{E})$.
\end{Definition}

In view of Proposition \ref{Proposition:Prop5.4}, we define now a generalization of the sets $\mathcal{E}_{e^j}$ in the language of graphs in the following way.

\begin{Definition}
\label{Definition:edge_set}
Let $\psi,\theta: \mathbb{N} \longrightarrow \mathbb{R}_{ \geqslant 0}$ be finitely supported, with $V_{\psi} = \supp \psi$ and $W_{\theta} = \supp \theta$, and let $t \geqslant 1$ and $C \in \mathbb{R}$. Then we define \[ \mathcal{E}_{\psi,\theta}^{t, C}: = \Bigg\{(v,w) \in V_{\psi} \times W_{\theta}: \, D_{\psi,\theta}(v,w) \leqslant 1, \, \sum_{\substack{p \vert \frac{vw}{\gcd(v,w)^2} \\ p \geqslant t}} \frac{1}{p}\geqslant C\Bigg\},\] where
\[ D_{\psi,\theta}(v,w) = \frac{\max\{w\psi(v), v \theta(w)\}}{\gcd(v,w)}.\] 
\end{Definition}

\noindent With this at hand, we state the following technical-looking main result from where we deduce Proposition \ref{Proposition:Prop5.4} immediately. 

\begin{Theorem}[Main technical result]
\label{Theorem:maintheorem}
Let $\varepsilon \in (0,2/5]$. Then there exists $p_0(\varepsilon) >0$ such that the following holds. Let $\psi,\theta: \mathbb{N} \longrightarrow \mathbb{R}_{\geqslant 0}$ be finitely supported o, $V_{\psi} = \supp \psi$, $W_{\theta} = \supp \theta$ and \[ \mathcal{P}_{\psi,\theta}: = \{p: \, \exists \, (v,w) \in V_{\psi}\times W_{\theta} \, \text{s.t.} \, p \vert vw\}.\] Let $P_{\psi,\theta}(\varepsilon): = p_0(\varepsilon) + \vert \mathcal{P}_{\psi,\theta} \cap [1, p_0(\varepsilon)]\vert$. 
Suppose that $\mathcal{E} \subset \mathcal{E}_{\psi, \theta}^{t, C} \cap (V \times W)$.
Then for all $t \geqslant 1$ and any $C \in \mathbb{R}$ we have
\begin{equation}
\label{eq:maintheoremconclusion}
\mu_{\psi,\theta}(\mathcal{E}) \leqslant 1000^{P_{\psi,\theta}(\varepsilon)}( \mu_{\psi}(V) \mu_{\theta}(W) e^{-Ct})^{\frac{1}{2} + \varepsilon}.
\end{equation}
\end{Theorem}

\begin{proof}[Proof of Proposition \ref{Proposition:Prop5.4} assuming Theorem \ref{Theorem:maintheorem}]
    Defining $\widetilde{\psi}(q) := \mathds{1}_{[q \in [X,Y]]}\frac{\psi(q)}{e^j}$, we observe that
    $\mathcal{E}_{e^j} = \mathcal{E}_{\widetilde{\psi},\widetilde{\psi}}^{e^j,10}$ (in the notation of Definition \ref{Definition:edge_set}). Thus applying 
    Theorem \ref{Theorem:maintheorem}, we obtain
    \begin{align*}
    \sum_{(v,w) \in \mathcal{E}_{e^j}} \frac{\psi(v)\varphi(v)}{v}\frac{\psi(w)\varphi(w)}{w} = (e^{2j})\mu_{\widetilde{\psi},\widetilde{\psi}}(\mathcal{E}_{\widetilde{\psi},\widetilde{\psi}}^{t,10}) &\ll_{\varepsilon} {e^j}^{1 - 2\varepsilon}e^{-10e^j\big(\tfrac{1}{2} + \varepsilon\big)}\ll_{\varepsilon} \frac{1}{e^j},
\end{align*}
as required. 
\end{proof}

\subsubsection{Proof of Theorem \ref{Theorem:maintheorem}}
We now apply the two-step strategy described on p. \pageref{strategy}. We believe that the first part (i.e. Proposition \ref{Prop:structureofminimalcounterexample}) is the heart of the method, and can in adapted form also be applied to other problems. Proposition \ref{Prop:resolvingminimalcounterexample} is more specific to this problem, and in an application in another area probably needs to be replaced by distinct arguments that exploit the structure obtained in Proposition \ref{Prop:structureofminimalcounterexample} in a different way.

\begin{Proposition}[Structure of minimal potential counterexample]
\label{Prop:structureofminimalcounterexample}
Suppose Theorem \ref{Theorem:maintheorem} were false. Fix functions $\psi,\theta: \mathbb{N} \longrightarrow \mathbb{R}_{\geqslant 0}$ with finite support such that $\vert \mathcal{P}_{\psi,\theta}\vert$ is minimal over all such pairs of functions for which there exist instances of $\varepsilon$, $t$, $C$, and $\mathcal{E} \subset \mathcal{E}^{t,C}_{\psi,\theta}$ satisfying the hypotheses of Theorem \ref{Theorem:maintheorem} for which \eqref{eq:maintheoremconclusion} fails. Fix such instances of  $\varepsilon$, $t$, $C$ and $\mathcal{E}$. Write $q^\prime: = \frac{2}{1 + 2\varepsilon}$. Then there exists $\mathcal{E}^{\prime} \subset \mathcal{E}$ for which
\begin{enumerate}
\item \emph{($\mathcal{E}^{\prime}$ is a near counterexample)}: $\mu_{\psi,\theta}(\mathcal{E}^{\prime}) > \frac{1}{2} \cdot 1000^{P_{\psi,\theta}(\varepsilon)}( \mu_{\psi}(V^{\prime}) \mu_{\theta}(W^{\prime}) e^{-Ct})^{\frac{1}{q^\prime}}$ where\\
 $V^{\prime}:= \mathcal{E}^{\prime}|_V, W^{\prime}: = \mathcal{E}^{\prime}|_W$.
\item \emph{($\mathcal{E}^{\prime}$ is combinatorially structured)}: for all $v \in V^{\prime}$ and $w \in W^{\prime}$, \[\mu_{\theta}(\Gamma_{\mathcal{E}^{\prime}}(v)) \geqslant \frac{1}{q^\prime} \frac{\mu_{\psi,\theta}(\mathcal{E}^{\prime})}{\mu_{\psi}(V^{\prime})} \qquad \text{and} \qquad \mu_{\psi}(\Gamma_{\mathcal{E}^{\prime}}(w)) \geqslant \frac{1}{q^\prime} \frac{\mu_{\psi,\theta}(\mathcal{E}^{\prime})}{\mu_{\theta}(W^{\prime})} .\] 
\item \emph{($\mathcal{E}^{\prime}$ is arithmetically structured)}: there exists $N \in \mathbb{N}$ such that for all primes $p$ and for all $(v,w) \in \mathcal{E}^{\prime}$, $\vert \nu_p(v/N)\vert + \vert \nu_p(w/N)\vert \leqslant 1.$ 
\end{enumerate}
\end{Proposition}

Proposition \ref{Prop:structureofminimalcounterexample}  combined with the following result provides immediately the proof of Theorem \ref{Theorem:maintheorem}.\\

\begin{Proposition}[Resolution of minimal potential counterexample]
\label{Prop:resolvingminimalcounterexample}
Fix $\psi,\theta, \varepsilon, C, t$ satisfying the hypotheses of Theorem \ref{Theorem:maintheorem}. Suppose $\mathcal{E}^{\prime} \subset \mathcal{E}^{t,C}_{\psi,\theta}$ satisfies properties (b) and (c) from Proposition \ref{Prop:structureofminimalcounterexample} (with these parameters). Then $\mathcal{E}^{\prime}$ cannot satisfy property (a) of Proposition \ref{Prop:structureofminimalcounterexample}. 
\end{Proposition}

    \subsubsection{Proof of Proposition \ref{Prop:structureofminimalcounterexample}}

   Take $\psi,\theta,\varepsilon,C,t,\mathcal{E},V,W$ as in the statement of Proposition \ref{Prop:structureofminimalcounterexample}, i.e.\! forming a minimal counterexample to Theorem \ref{Theorem:maintheorem}. For a fixed prime $p \in \mathcal{P}_{\psi,\theta}$, we now partition $V$ and $W$ depending on whether $p$ divides the elements or not (recall the simplifying square-free assumption we made): Let
   \begin{align*}V_{1}&:= \{v \in V: \, p\mid v\},\quad \hspace{4mm}V_{0}:= \{v \in V: \, p\nmid v\},\\
   W_{1}&:= \{w \in W: \, p\mid w\},\quad W_{0}:= \{w \in W: \, p\nmid w\}.\end{align*}
   
   For $i,j \in \{0,1\}$, we now define
   \[ m_p(i,j): = \frac{\mu_{\psi,\theta}(\mathcal{E} \cap (V_i \times W_j))}{\mu_{\psi,\theta}(\mathcal{E})}\] and observe that by $m_p(i,j) \geqslant 0$ and \[ \sum_{i,j \geqslant 0} m_p(i,j) = 1,\] $m_p$ defines a probability measure on $\{0,1\}^2$. 

Using the minimality assumption on $\mathcal{E}$, we now upper-bound $m_p(i,j)$ by a bilinear expression with exponential decay away from the diagonal $i=j$.\\

\begin{Lemma}[Bilinear upper bound]
\label{Lemma:bilinearbound}
Let $\psi,\theta,\varepsilon,C,t,\mathcal{E},m_p(i,j)$ be as above, and define H\"older conjugates $q,q'$ (i.e. $1/q + 1/q' =1$) with $q = 2+ \varepsilon$.
Writing for $i,j \geqslant 0$, $\alpha_i: = \mu_{\psi}(V_i)/ \mu_{\psi}(V)$ and $\beta_j: = \mu_{\theta}(W_j)/\mu_{\theta}(W)$, we have
\begin{equation}\label{bilinear_bound_prime} m_p(i,j) \leqslant 1000^{-\mathds{1}_{[p \leqslant p_0(\varepsilon)]}}  p^{-\frac{\vert i-j\vert}{q}}   (\alpha_i \beta_j e^{ \mathds{1}_{[i \neq j]}}) ^{\frac{1}{q'}}.\end{equation}
\end{Lemma}

\begin{proof}
We upper-bound $m_p(i,j)$ by bounding $\mu_{\psi,\theta}(\mathcal{E})$ from below and $\mu_{\psi,\theta}(\mathcal{E} \cap (V_i \times W_j))$ from above. The na\"{i}ve combination of these bounds will control $m(i,j)$ as required. 

Indeed, since $\mathcal{E}$ is a counterexample to Theorem \ref{Theorem:maintheorem} we have \begin{equation}
\label{eq:lowerboundonE}
\mu_{\psi,\theta}(\mathcal{E}) \geqslant 1000^{P_{\psi,\theta}(\varepsilon)}( \mu_{\psi}(V) \mu_{\theta}(W) e^{-Ct})^{\frac{1}{2} + \varepsilon}.
\end{equation} In order to bound $\mu_{\psi,\theta}(\mathcal{E} \cap (V_i \times W_j))$ from above we will remove the influence of the prime $p$ and then use the minimality assumption. The main barrier is notational.  Define $\widetilde{\psi_{i,j}}, \widetilde{\theta_{i,j}}: \mathbb{N} \longrightarrow \mathbb{R}_{\geqslant 0}$ by \[\widetilde{\psi_{i,j}}(v) := \begin{cases}
p^{j - \min(i,j)}\psi(p^i v) & \text{if } \gcd(p,v) = 1, \\
0 & \text{if } p \vert v, \end{cases}\] and\[\widetilde{\theta_{i,j}}(w) := \begin{cases}
p^{i - \min(i,j)}\theta(p^jw) & \text{if } \gcd(p,w) = 1, \\
0 & \text{if } p \vert w, \end{cases}\] with
 \begin{align*}
\widetilde{\mathcal{E}_{i,j}} := \{(v,w): (p^i v, p^j w) \in \mathcal{E} \cap (V_i \times W_j)\},\quad \widetilde{V_i}:= \{ v: \, p^iv \in V_i\},\quad
\widetilde{W_j}:= \{w: \, p^j w \in W_j\}.
\end{align*} 

Elementary calculations show

\[\mu_{\widetilde{\psi_{i,j}}}(\widetilde{V_i}) = p^{j- \min(i,j)} \Big(\frac{p^i}{\varphi(p^i)}\Big)\mu_{\psi}(V_i), \quad \mu_{\widetilde{\theta_{i,j}}}(\widetilde{W_j}) = p^{i - \min(i,j)} \Big( \frac{p^j}{\varphi(p^j)}\Big)\mu_{\theta}(W_j),\]

as well as
\begin{align*}
\mu_{\psi,\theta}(\mathcal{E} \cap (V_i \times W_j))
&=\Big( \frac{ \varphi(p^i) \varphi(p^j)}{p^{i+j}}\Big)p^{-\vert i-j\vert} \mu_{\widetilde{\psi_{i,j}}, \widetilde{\theta_{i,j}}}(\widetilde{\mathcal{E}_{i,j}}).
\end{align*}
Further, a quick check shows us that $\widetilde{\mathcal{E}_{i,j}} \subset \mathcal{E}^{t, C - \frac{\mathds{1}_{[i \neq j]}}{t}}_{\widetilde{\psi_{i,j}}, \widetilde{\theta_{i,j}}}$ and $\mathcal{P}_{\widetilde{\psi_{i,j}}, \widetilde{\theta_{i,j}}} = \mathcal{P}_{\psi,\theta} \setminus \{p\}$. Hence $\widetilde{\mathcal{E}_{i,j}}$ is of the type that is considered in Theorem \ref{Theorem:maintheorem} with one less prime considered, enabling us to use the minimality assumption on $\mathcal{E}$ to upper-bound the contribution from $\widetilde{\mathcal{E}_{i,j}}$.
Therefore, by the minimality assumption on $\vert \mathcal{P}_{\psi,\theta}\vert$, the bound \eqref{eq:maintheoremconclusion} holds for $\mu_{\widetilde{\psi_{i,j}}, \widetilde{\theta_{i,j}}}(\widetilde{\mathcal{E}_{i,j}})$, which provides by the above equations an upper bound for $\mu_{\psi,\theta}(\mathcal{E} \cap (V_i \times W_j))$.
Combining this with the lower bound on $\mu_{\psi,\theta}(\mathcal{E})$ from equation \eqref{eq:lowerboundonE}, we obtain 
\eqref{bilinear_bound_prime}.
\end{proof}

Having the bilinear bound \eqref{bilinear_bound_prime} on $m_p(i,j)$ established, 
the factor $1000^{-\mathds{1}_{[p \leqslant p_0(\varepsilon)]}}$ in \eqref{bilinear_bound_prime} allows us to neglect all primes $p \leqslant p_0(\varepsilon)$. Further, from \eqref{bilinear_bound_prime} 
one can deduce by several usages of the Cauchy--Schwarz and H\"older inequality (see \cite[Lemma 3.2]{HSW} or \cite[Lemma 2.1]{GW21} for details) that we have a strong concentration point of the measure in the following sense:
There exists $k_p \in \{0,1\}$ such that
\[\sum\limits_{i \neq k_p, j \neq k_p} m_p(i,j) \ll p^{-1 -2\varepsilon} + p^{-11/10}.\]
By considering all primes $p \in \mathcal{P}_{\psi,\theta}$ (that are at least $p_0(\varepsilon)$), and taking the union bound, we conclude that 
\begin{align*}
&\mu_{\psi,\theta}( \{(v,w) \in \mathcal{E}: \, \exists p \in \mathcal{P}_{\psi,\theta} \, \text{s.t.} \, \vert \nu_p(v) - k_p \vert + \vert \nu_p(w) - k_p\vert \geqslant 2\})\nonumber\\
& \ll_{\varepsilon} \mu_{\psi, \theta}(\mathcal{E}) \sum_{p > p_0(\varepsilon)} (p^{-1 - 2 \varepsilon} + p^{-\frac{11}{10}}).\nonumber
\end{align*}
Since $p_0(\varepsilon)$ is sufficiently large, the converging sum above can be chosen to be $< 1/2$, so we obtain
\begin{equation}
\label{eq:removingbadpairs}
\mu_{\psi,\theta}( \{(v,w) \in \mathcal{E}: \, \exists p \in \mathcal{P}_{\psi,\theta} \, \text{s.t.} \, \vert \nu_p(v) - k_p \vert + \vert \nu_p(w) - k_p\vert \geqslant 2\}) <\frac{\mu_{\psi, \theta}(\mathcal{E})}{2}. 
\end{equation}

In other words, setting $N := \prod_{p \in \mathcal{P}_{\psi,\theta}}p^{k_p}$, we have
that at least $50\%$ of the measure is concentrated on pairs $(v,w) \in \mathcal{E}$ where the $p$-adic valuations satisfy \eqref{padic_structure}.
So we define \[ \mathcal{E}^{*}: = \{(v,w) \in \mathcal{E}: \, \text{for all primes } p, \, \vert \nu_p(v/N)\vert + \vert \nu_p(w/N)\vert \leqslant 1\},\]
$V^*: = \mathcal{E}^*|_V$, $W^*: = \mathcal{E}^*|_W$. Note that \eqref{eq:removingbadpairs} implies 
\begin{align*}
\mu_{\psi,\theta}(\mathcal{E}^{*}) \geqslant \frac{\mu_{\psi, \theta}(\mathcal{E})}{2} &> \frac{1}{2} \cdot 1000^{P_{\psi,\theta}(\varepsilon)}( \mu_{\psi}(V) \mu_{\theta}(W) e^{-Ct})^{\frac{1}{q^\prime}} \\
& \geqslant \frac{1}{2} \cdot 1000^{P_{\psi,\theta}(\varepsilon)}( \mu_{\psi}(V^{*}) \mu_{\theta}(W^{*}) e^{-Ct})^{\frac{1}{q^\prime}}.
\end{align*}
Now define $\mathcal{E}^{\prime} \subset \mathcal{E}^{*}$ to be minimal (with respect to inclusion) 
such that property (a) of Proposition \ref{Prop:structureofminimalcounterexample} is satisfied. By construction, we have that $\mathcal{E}^{\prime}$ exists (since $\mathcal{E}^*$ is finite and satisfies property (a) itself). $\mathcal{E}^{\prime}$ also satisfies property (c) (since $\mathcal{E}^*$ satisfies property (c)). The minimality assumption of 
$\mathcal{E}^{\prime}$ now gives us the remaining property (b) by the ``popularity principle'': Assuming (b) not to hold, then there must exist an element $v \in V^{\prime}$ (the argument for $W^{\prime}$ is the same) for which 
\begin{equation*}
\mu_{\theta}(\Gamma_{\mathcal{E}^{\prime}}(v)) < \frac{1}{q^\prime} \frac{\mu_{\psi,\theta}(\mathcal{E}^{\prime})}{\mu_{\psi}(V^{\prime})}.
\end{equation*}
Removing the vertex $v$ from the graph, we can check straightforwardly that for the arising sub-graph, both properties (a) and (c) remain true, a contradiction to the minimality of $\mathcal{E}^{\prime}$.

 \subsubsection{Proof of Proposition \ref{Prop:resolvingminimalcounterexample}}

We need the following two auxiliary ingredients that can be seen as statements on the ``anatomy of integers'', and are of the form of standard averaging techniques in analytic number theory we mentioned on p. \pageref{avg_techn}. The proof methodology here is not novel (they follow from applying standard tricks in analytic number theory such as the exponential Markov inequality/the exponential Rankin trick), so we omit them here. The interested reader might consult \cite[Lemma 10]{ABH23} and \cite[Lemma 4.2]{HSW} for the proofs.

\begin{Lemma}[Unweighted anatomy property]
\label{Lemma:unweightedanatomy}
For any real $x,t \geqslant 1$ and $c \in \mathbb{R}$, 
\[\# \Big\{n \leqslant x: \, \sum\limits_{\substack{p \geqslant t \\ p \vert n}} \frac{1}{p} \geqslant c \Big\} \ll x e^{-100ct}.\]
\end{Lemma}

\begin{Lemma}[Divisor anatomy property]
\label{Lemma:divisoranatomy}
For any $M \in \mathbb{N}$, real $t \geqslant 1$ and $c \in \mathbb{R}$, \[\sum\limits_{\substack{mn = M \\ \sum\limits_{\substack{p \geqslant t \\ p\vert m}} \frac{1}{p} \geqslant c}} \varphi(n) \ll Me^{-100ct}.\]
\end{Lemma}

\begin{proof}[Proof of Proposition \ref{Prop:resolvingminimalcounterexample}]
Let $\psi,\theta,\varepsilon,C,t,\mathcal{E}^{\prime}$ be as in the statement of Proposition \ref{Prop:resolvingminimalcounterexample}, and $N$ be the natural number from property (c) of Proposition \ref{Prop:structureofminimalcounterexample} that $\mathcal{E}^{\prime}$ satisfies. We will prove that 
\begin{equation}
\label{eq:halfbound}
\mu_{\psi,\theta}(\mathcal{E}^{\prime}) \ll (\mu_{\psi}(V^{\prime}) \mu_{\theta}(W^{\prime}) e^{-25 Ct})^{\frac{1}{2}}.
\end{equation}

To ``smuggle in'' the additional power of $\varepsilon$ in order to show \[\mu_{\psi,\theta}(\mathcal{E}^{\prime}) \leq (\mu_{\psi}(V^{\prime}) \mu_{\theta}(W^{\prime}) e^{-Ct})^{\frac{1}{2} + \varepsilon},\]
one only needs to show that $\mu_{\psi}(V^{\prime}) \mu_{\theta}(W^{\prime})$ is not ``too small'' - in which case even the trivial bound $\mu_{\psi,\theta}(\mathcal{E}^{\prime}) \leqslant \mu_{\psi}(V^{\prime})\mu_{\theta}(W^{\prime})$ contradicts assumption (a) of Proposition \ref{Prop:structureofminimalcounterexample}. So one is left to show \eqref{eq:halfbound}.

 In order to do this, we define for each $(v,w) \in V^{\prime} \times W^{\prime}$
\begin{align*}
v^{-} &:= \frac{N}{\gcd(N,v)}, \quad v^{+} := \frac{v}{\gcd(N,v)}, \quad\\
w^{-} &:= \frac{N}{\gcd(N,w)}, \quad w^{+} := \frac{w}{\gcd(N,w)}.
\end{align*}
Since $\mathcal{E}^{\prime}$ satisfies property (3), if $(v,w) \in \mathcal{E}^{\prime}$ we have (recall our square-free assumption) that all four of $v^-, v^+, w^-, w^+$ are pairwise coprime, as well as $v^+,w^+$ being coprime to $N$,
and \[v = N\frac{v^+}{v^-}, \qquad w = N \frac{w^+}{w^-}, \quad \frac{vw}{\gcd(v,w)^2} = v^-v^{+}w^{-}w^{+}.
\] Therefore, as $(v,w) \in \mathcal{E}^{t,C}_{\psi,\theta}$, \[ \psi(v) \leqslant \frac{ \gcd(v,w)}{w} = \frac{1}{v^{-} w^{+}}, \qquad \theta(w) \leqslant \frac{\gcd(v,w)}{v} = \frac{1}{v^{+} w^{-}}.\] Let $w_0 \in W^{\prime}$ maximize $w_0^+$, and let $v_0(w) \in \Gamma_{\mathcal{E}^{\prime}}(w)$ maximize $v_0^+(w)$ (over $\Gamma_{\mathcal{E}^{\prime}}(w)$). 
Then the above inequalities are sharpened by
\begin{equation}
\label{eq:upper_bound_psi}
\psi(v) \leqslant  \frac{1}{v^{-} w_0^{+}}, \qquad \theta(w) \leqslant \frac{1}{v^{+}_0(w) w^{-}}
\end{equation} for  $(v,w) \in \mathcal{E}^{\prime}$ with $(v,w_0) \in \mathcal{E}^{\prime}$.
By applying property (b) twice, we get
\begin{align}
\label{eq:qualitybounding}
\frac{\mu_{\psi,\theta}(\mathcal{E}^{\prime})}{(\mu_{\psi}(V^{\prime}) \mu_{\theta}(W^{\prime}))^{\frac{1}{2}}}
\leqslant q^\prime \Big( \sum\limits_{v \in \Gamma_{\mathcal{E}^{\prime}}(w_0)} \mu_{\psi}(v) \mu_{\theta}(\Gamma_{\mathcal{E}^{\prime}}(v))\Big)^{\frac{1}{2}},
\end{align}
and by \eqref{eq:upper_bound_psi} we obtain 
\begin{align*}
 \sum\limits_{v \in \Gamma_{\mathcal{E}^{\prime}}(w_0)} \mu_{\psi}(v) \mu_{\theta}(\Gamma_{\mathcal{E}^{\prime}}(v))& = \sum\limits_{v \in \Gamma_{\mathcal{E}^{\prime}}(w_0)} \frac{\varphi(v) \psi(v)}{v} \sum_{w \in \Gamma_{\mathcal{E}^{\prime}}(v)} \frac{ \varphi(w) \theta(w)}{w} \nonumber\\
&\leqslant \frac{1}{w_0^+} \sum_{w \in W^{\prime}} \frac{ \varphi(w)}{w w^-} \cdot \frac{1}{v_0^+(w)} \sum\limits_{v \in \Gamma_{\mathcal{E}^{\prime}}(w)} \frac{\varphi(v)}{v v^-}.
\end{align*}

Since this is a sum over all pairs $(v,w) \in \mathcal{E}^{\prime} \subset \mathcal{E}^{t,C}_{\psi,\theta}$, and $\frac{vw}{\gcd(v,w)^2} = v^- v^+ w^- w^+$, we know that \[ \sum\limits_{\substack{ p \geqslant t \\ p \vert v^- v^+ w^- w^+}} \frac{1}{p} \geqslant C.\] Hence $\sum\limits_{\substack{p \geqslant t: \, p \vert v^-}} \frac{1}{p} \geqslant \frac{C}{4}$ or similarly with $p\vert v^{+}$, $p\vert w^{-}$ or $p \vert w^+$. Therefore \[\sum\limits_{v \in \Gamma_{\mathcal{E}^{\prime}}(w_0)} \mu_{\psi}(v) \mu_{\theta}(\Gamma_{\mathcal{E}^{\prime}}(v))  \leqslant S_1 + S_2 + S_3 + S_4,\] where 
\begin{equation}
\label{eq:S1definition}
S_1 := \frac{1}{w_0^+} \sum_{w \in W^{\prime}} \frac{ \varphi(w)}{w w^-} \cdot \frac{1}{v_0^+(w)} \sum\limits_{\substack{v \in \Gamma_{\mathcal{E}^{\prime}} (w)\\ \sum\limits_{\substack{p \geqslant t \\ p \vert v^-}} \frac{1}{p} \geqslant \frac{C}{4}}} \frac{\varphi(v)}{v v^-},
\end{equation}
and $S_2, S_3, S_4$ are similar expressions with the ``anatomy'' condition placed on $p\vert v^+$, $p\vert w^-$, and $p\vert w^+$ respectively. \\

\textbf{Bounding $S_1$}: We bound the inner sum from \eqref{eq:S1definition}. Since $v = v^+ \cdot \frac{N}{v^-}$ and $v^+ \leqslant v_0^+(w)$, we can reparametrise $v$ in terms of variables $v^+$ and $v^-$ and obtain
\begin{align}
\label{eq:vplusvminus}
\sum\limits_{\substack{v \in \Gamma_{\mathcal{E}^{\prime}} (w)\\ \sum\limits_{\substack{p \geqslant t \\ p \vert v^-}} \frac{1}{p} \geqslant \frac{C}{4}}} \frac{\varphi(v)}{v v^-} &\leqslant \sum\limits_{\substack{v^+ \leqslant v_0^+(w)}} \sum\limits_{\substack{v^- \vert N \\ \gcd(v^-, v^+) = 1 \\ \sum\limits_{\substack{p\geqslant t \\ p \vert v^-}} \frac{1}{p} \geqslant \frac{C}{4}}} \frac{ \varphi(v^+ \cdot \frac{N}{v^-})}{Nv^+}.
\end{align}
Since $\varphi$ is multiplicative and $\varphi(n) \leq n$, we have 
\begin{align*}
\frac{ \varphi(v^+ \cdot \frac{N}{v^-})}{Nv^+} & =\frac{ \varphi(v^+)}{v^+}\cdot \frac{\varphi(\tfrac{N}{v^{-}})}{N}\leqslant \frac{\varphi(\frac{N}{v^-})}{N}.
\end{align*}
Thus, the right-hand side of \eqref{eq:vplusvminus} is by Lemma \ref{Lemma:divisoranatomy}
\begin{align}
\label{eq:similar_to_DS}
\leqslant \sum\limits_{v^+ \leqslant v_0^+(w)} \Big( \frac{1}{N} \sum\limits_{\substack{v^- \vert N \\ \sum\limits_{\substack{p\geqslant t \\ p \vert v^-}} \frac{1}{p} \geqslant \frac{C}{4}}} \varphi(\tfrac{N}{v^-})\Big)\leqslant  \sum\limits_{\substack{v^+ \leqslant v_0^+(w)}}  \Big( \frac{1}{N} \sum\limits_{\substack{mn = N \\ \sum\limits_{\substack{p \geqslant t \\ p \vert m}} \frac{1}{p} \geqslant \frac{C}{4}}} \varphi(n)\Big) \ll v_0^+(w) e^{-25Ct}.
\end{align}
 Therefore, by \eqref{eq:S1definition} we get \[S_1 \ll \frac{e^{-25Ct}}{w_0^+} \sum_{w \in W^{\prime}} \frac{\varphi(w)}{w w^-} \ll \frac{e^{-25Ct}}{w_0^+} \sum\limits_{\substack{w^+ \leqslant w_0^+}} \sum\limits_{\substack{w^- \vert N \\ \gcd(w^-, w^+) = 1}} \frac{\varphi(w^+ \cdot \frac{N}{w^-})}{N w^+} \ll e^{-25Ct},\] where the final step followed from the method we used to bound expression \eqref{eq:vplusvminus}, using $\sum_{d \mid n}\varphi(d) = n$ instead of Lemma \ref{Lemma:divisoranatomy}. Hence $S_1 \ll e^{-25Ct}$. \\

The bound of $S_2$ follows a similar path, with Lemma \ref{Lemma:unweightedanatomy} in place of Lemma \ref{Lemma:divisoranatomy}. By symmetry, the bounds of $S_3$ and $S_4$ follow in the exactly same ways,
which finally implies
$S_1 + S_2 + S_3 + S_4 \ll e^{-25 Ct}$. Substituting this bound into \eqref{eq:qualitybounding} and using $q^\prime < 2$, we have \[ \frac{\mu_{\psi,\theta}(\mathcal{E}^{\prime})}{(\mu_{\psi}(V^{\prime}) \mu_{\theta}(W^{\prime}))^{\frac{1}{2}}} \ll e^{-\frac{25}{2} Ct},\] resolving equation \eqref{eq:halfbound}, which in turn concludes the proof of Proposition \ref{Prop:resolvingminimalcounterexample}.
\end{proof}
 \phantom\qedhere
\end{proof}
\end{enumerate}

\subsection*{Acknowledgements}

The author thanks the Austrian Mathematical Society for awarding him the
ÖMG Studienpreis 2024 for his dissertation under the supervision of C. Aistleitner. Further, the author is grateful to C. Fuchs for inviting him to write this article for publication in IMN.\\

\bibliographystyle{plain}
\bibliography{bibliography.bib}
\end{document}